\documentclass[12pt]{amsart}
\usepackage{amsmath}
\usepackage{amsfonts}
\usepackage[mathscr]{euscript}
\usepackage{amsthm}
\usepackage{amssymb}
\usepackage{float}
\usepackage[english]{babel}
\usepackage[utf8]{inputenc}
\usepackage[T1]{fontenc}
\usepackage{tikzsymbols}

\usepackage[all]{xy}
\usepackage{color}

\usepackage{scalerel}[2014/03/10]
\usepackage{stackengine}

\DeclareMathOperator{\Rep}{Rep}     
\DeclareMathOperator{\Spec}{Spec}   

\newcommand{\bN}{\mathbb{N}}        

\newcommand{\sF}{\mathcal{F}}       
\newcommand{\sH}{\mathcal{H}}       
\newcommand{\sO}{\mathcal{O}}       
\newcommand{\sT}{\mathcal{T}}       

\renewcommand{\geq}{\geqslant}      

\theoremstyle{definition}     

\newtheorem{defi}{Definition}[section]
\newtheorem{remark}[defi]{Remark}

\newtheorem{definition}[defi]{Definition}

\theoremstyle{plain}

\newtheorem{theorem}[defi]{Theorem}
\newtheorem{corollary}[defi]{Corollary}
\newtheorem{lemma}[defi]{Lemma}
\newtheorem{proposition}[defi]{Proposition}

\usepackage{geometry}
\input xy
\xyoption{all}

\title{Extension of torsors and prime to $p$ fundamental group scheme}
\author{Marco Antei}
\author{Jimmy Calvo-Monge}

\begin{document}

\maketitle
\begin{abstract}
    Let $R$ be a discrete valuation ring. Let $X$ be a proper and faithfully flat $R$-scheme, endowed with a section $x \in X(R)$, with integral and normal fibres. Let $f: Y \rightarrow X_{\eta}$ be a finite Nori-reduced $G$-torsor. In this paper we provide a useful criterion to extend $f: Y \rightarrow X_{\eta}$ to a torsor over $X$. Furthermore in the particular situation where $R$ is a complete discrete valuation ring of residue characteristic $p>0$ and $X\to \Spec(R)$ is smooth we apply our criterion to prove that the natural morphism 
    $\psi^{(p')}: \pi(X_{\eta},x_{\eta})^{(p')}\to \pi(X,x)_{\eta}^{(p')}
$
between the prime-to-$p$ fundamental group scheme of $X_{\eta}$ and the generic fibre of the prime-to-$p$ fundamental group scheme of $X$ is an isomorphism. This generalizes a well known result for the étale fundamental group. The methods used are \emph{purely} tannakian.
\end{abstract}

    
    \tableofcontents
    
\section{Introduction}
Let $X$ be a finite type scheme, faithfully flat over the spectrum of a Dedekind ring $R$ with fraction field $K$. Let $X_{K}$ denote its generic fibre and let $f:Y\to X_{K}$ be a $G$-torsor, with $G$ an affine and algebraic $K$-group scheme. The question of finding a \emph{model} for $f:Y\to X_{K}$, i.e. a torsor over $X$ whose generic fibre is isomorphic to  $f$, is as old as the famous \emph{Théorie de la spécialisation du groupe fondamental} introduced by Grothendieck in \cite[Expos\'e X]{SGA}. An historical overview of some of the most recent attempts and results toward a general solution can be found in many of the papers mentioned in the bibliography and several others, so it will not be repeated in this introduction. 

In this article we are providing an useful criterion to extend finite pointed torsors. In few words we prove that in order for the torsor $Y\to X_{K}$ to be extended to $X$ it is sufficient (and uninterestingly necessary) to extend it with a finite and faithfully flat morphism $T\to X$ with some few extra assumptions which are often satisfied even when $T\to X$ is not a torsor. To show the \emph{usefulness} of our criterion we will provide an interesting application of our main result to the base change behavior of the fundamental group scheme, as defined in  \cite{Nori76}, \cite{Nori82} and \cite{AEG20} (cf. also \cite{HdS19} for a more recent construction). This will be explained at the end of this introduction. The malleability of our criterion will also hopefully lead to new \emph{extension} results in the near future.

As usual throughout the whole paper we will denote by $X_B$ (resp. $g_B$) the pullback over the $A$-algebra $B$ of the $A$-scheme $X$ (resp. of the morphism between $X$-schemes $g$), whenever $A$ is a ring. Also, whenever we will consider a $G$-torsor $f: Y \to X_K$, we will always assume, if not stated otherwise, that $G$ is finite, that $f: Y \to X_K$ is Nori-reduced (or \emph{Galois}), meaning that $\mathcal{O}_Y(Y)=K$, and that   it comes with a point $y \in Y(\Spec(K))$ over $x_K$ where $x\in X(R)$ is a given section. Finally we will always assume $X$ to be proper over $R$ and that its fibres are integral and normal.  Under these assumptions we will prove the following result. (See the next figure for an illustration).

\begin{theorem}\label{teoGENERAL}
Let $R$ be a complete DVR (i.e. discrete valuation ring). Suppose there exist a finite field extension $K^{\prime}/K$ and a finite and faithfully flat morphism $\varphi: Z \rightarrow X$ such that $\varphi^{*}\varphi_{*}\mathcal{O}_Z$ is a free $\mathcal{O}_Z$-module and $\varphi_K= \lambda \circ f_{K^{\prime}}$ where $\lambda: X_{K'} \to X_K$ and $f_{K^{\prime}}:Y_{K^{\prime}}\to X_{K'}$ are the natural pullback morphisms. If moreover $\sO_Z(Z)=R'$ where $R'$ is the integral closure of $R$ in $K'$ then, there is a $M$-torsor $f_1: Y_1 \rightarrow X$, for some quasi-finite affine and flat group scheme $M$ over $R$, whose generic fibre equals the $G$-torsor $f: Y \rightarrow X_K$.
\end{theorem}

 
\begin{figure}[H]
    \centering
    \begin{align}
    \xymatrixrowsep{0.3in}
        \xymatrix{
        Y^{\prime} \ar[d]_{f_{K^{\prime}}} \ar[rd]_{\alpha} \ar[r]^{\widetilde{j}} & Z \ar[ddr]^{\varphi} & \\
        X_{K^{\prime}}  \ar[d]_{p_{K^{\prime}}} \ar[rd]_{\lambda} & Y \ar[d]_f  &  \\
        \text{Spec}(K^{\prime}) \ar[rd] & X_K  \ar[r]_j \ar[d]_{p_K} & X \ar[d]^p \\
        & \text{Spec}(K) \ar[r] & \text{Spec}(R).
        }
    \end{align}
    \caption{Hypotheses of Theorem \ref{teoGENERAL}.}
    \label{figura1}
\end{figure}

To prove this the reader will see that the hypothesis on $R$ being a complete DVR will be needed only to ensure that $R'$, the integral closure of $R$ in $K'$, is also a DVR. In consequence, if we find ourselves in the situation of Theorem \ref{teoGENERAL} with $K=K'$ then we would have that $R'=R$ is indeed a DVR and thus the hypothesis that $R$ is complete can be dropped. This gives us the following corollary, which is often useful when we already know how to \emph{close} $Y\to X_K$ before pulling it up over some finite field extension.

\begin{corollary}\label{corGENERAL}
Let $R$ be a (not necessarily complete) DVR and $X$, $Y$, $x$ as in Theorem \ref{teoGENERAL}. Assume $Y\to X_K$ can be extended to some finite and faithfully flat morphism $\varphi: Z \rightarrow X$ such that  $\varphi^{*}\varphi_{*}\mathcal{O}_Z$ is a free $\mathcal{O}_Z$-module. Then there exists a $M$-torsor $f_1: Y_1 \rightarrow X$, for some quasi-finite affine and flat group scheme $M$ over $R$, whose generic fibre equals the $G$-torsor $f: Y \rightarrow X_K$.
\end{corollary}
The proof of Theorem \ref{teoGENERAL} will be given in \S \ref{sec:proof}. It relies on the construction of some Tannakian lattice over $X$ that we will describe in \S \ref{sec2.2}. This approach is similar to the one taken in \cite{HdS19}, where Tannakian duality is considered for categories of coherent sheaves trivialized by proper surjective $H^0$-flat morphisms. For the comfort of the reader the notion of Tannakian lattice will be recalled in \S \ref{sec:2.1}

In section \S \ref{sez:PrimetoP} we will provide an interesting application of our previous result to the theory of fundamental group scheme. Again from the aforementioned Grothendieck's specialization theory, which will be widely recalled at the beginning of \S \ref{sez:PrimetoP}, it is known that, whenever $R$ is a complete discrete valuation ring with positive residue characteristic $p$, $\eta$ being the generic point of $\Spec(R)$, and with $X$ proper and smooth over $R$, then the morphism 
\begin{equation}\label{isoGRO}
    \varphi^{(p')}:\pi^{\text{\'et}}(X_{\overline{\eta}},\overline{x})^{(p')}\to \pi^{\text{\'et}}(X,\overline{x})^{(p')}
\end{equation}
    between the prime-to-$p$ fundamental groups is an isomorphism (see \ref{teoSPECIAL1} for more details and notations). A similar result for the fundamental group scheme, though \emph{expected}, has never been proved. In this article we will then prove that the morphism $$\psi^{(p')}: \pi(X_{\eta},x_{\eta})^{(p')}\to \pi(X,x)_{\eta}^{(p')}
$$
 is an isomorphism (cf. Theorem \ref{teoFUND} for a precise statement). In order to prove this claim one should be able to extend a pointed étale prime-to-$p$ torsor from $X_{\eta}$ to $X$ \emph{not} after extending scalars but over $R$ itself. However from isomorphism (\ref{isoGRO}) we can notably deduce that every prime-to-$p$ torsor over $X_{\eta}$ can be extended to a torsor over $X$ after extending scalars (for each torsor a new and different scalar extension thus giving rise to possibly infinitely many scalar extensions). To solve this inconvenient is where we will need our Theorem \ref{teoGENERAL}.

\section{The Tannakian lattice of trivialized vector bundles}\label{sec:2}
\subsection{A short introduction to Tannakian lattices}\label{sec:2.1}
To begin this section we present the essential definitions and the main results on Tannakian lattices as given by \cite{DH18}. First we discuss some notations. If $\varphi: R \to S$ is a ring homomorphism and $\mathcal{C}$ is a $R$-linear category, we define the category $\mathcal{C}_S$ obtained from $\mathcal{C}$ by extension of scalars as the category whose objects are the same as those of $\mathcal{C}$ and with morphisms given by $$ \text{Hom}_{\mathcal{C}_S}(X,Y)= \text{Hom}_{\mathcal{C}}(X,Y) \otimes_R S,$$
for any two objects $X,Y$ of $\mathcal{C}$. If $F: \mathcal{C} \rightarrow \mathcal{D}$ is an $R$-linear functor between two $R$-additive categories, then there is a natural extension of scalars $F_S: \mathcal{C}_S \rightarrow \mathcal{D}_S$ of the functor $F$. If both $\mathcal{C}, \mathcal{D}$ are tensor categories (over $R$), then also the categories $\mathcal{C}_S,\mathcal{D}_S$ are tensor categories (over $S$), and if $F$ is a tensor functor, then so is $F_S$.
In a tensor category $\mathcal{T}$ we denote by $\mathcal{T}^{\text{triv}}$, the full subcategory of trivial objects: it consists of objects $X$ of $\mathcal{T}$ admitting a monomorphism $X \rightarrow I^{\oplus n}$, where $n \in \mathbb{N}$ and $I$ is the unit object in $\mathcal{T}$. If $R$ is a Dedekind ring, we denote by $K$ its field of fractions.

\begin{remark}
In this paper we use the notion of trivialization meaning the following: a sheaf of modules $\sF$ over a scheme $X$ is said to be trivialized by a morphism $\phi: Y \to X$ if the inverse image $\phi^*\sF$ is isomorphic (as $\sO_Y$-modules) to $\sO_Y^{\oplus n}$ for some $n \in \bN$. This is not to be confused with the general notion of trivial objects in a category that we have presented above. The first meaning of trivialization, that of a trivialized sheaf of modules, was popularized in \cite{BdS11} and since it has been commonly used to refer to this phenomenon.
\end{remark}

\begin{definition}
A (neutral) Tannakian lattice over a Dedekind ring $R$ is an $R$-additive, rigid tensor category $\mathcal{T}$ that satisfies the following: 
\begin{enumerate}
\item  every morphism in $\mathcal{T}$ has kernel and image;
\item the category  $\mathcal{T}_K$ obtained by scalar extension is abelian;
\item there is an $R$-linear additive functor
\begin{align*}
\omega: \mathcal{T} \rightarrow \text{Mod}(R)
\end{align*}
such that 
\begin{itemize}
\item $\omega$ is faithful and preserves kernels and images;
\item the restriction of $\omega$ to $\mathcal{T}^{\text{triv}}$ is fully faithful.
\end{itemize}
\end{enumerate}
\end{definition}
A result on duality for tannakian lattices is presented in \cite{DH18}. We recall it hereafter for the comfort of the reader. In what follows, for a commutative ring with identity $A$, and $G= \Spec(B)$ an affine group scheme (finite or not) over $A$, the category $\Rep_A^0(G)$ (resp. $\Rep_A(G)$) consists of finite and projective (resp. finite) $A$-linear left $B$-comodules.

\begin{theorem}\label{2.3.2DH} \cite[Theorem 2.3.2]{DH18}
Let $(\mathcal{T}, \omega)$ be a Tannakian lattice over a Dedekind ring $R$. Then the group scheme $G = \textbf{Aut}_{\otimes}^R(\omega)$ is faithfully flat over $R$ and $\omega$ induces an equivalence between $\mathcal{T}$ and $\text{Rep}_R^0(G)$, the category of finite projective comodules over $G$.
\end{theorem}

\subsection{A Tannakian lattice over $X$}\label{sec2.2}

Throughout this section $X$ will be a proper and faithfully flat scheme over $R$, a complete DVR (as in figure \ref{figura1}) with integral and normal fibers, endowed with a section $x\in X(R)$; in particular $\mathcal{O}_X(X)=R$.
In order to tackle the problem of finding an extension of a given Galois torsor, first we can use Tannakian duality to obtain a group scheme over $R$ that eventually will lead to the construction of the desired torsor.
As usual we will denote by $\text{Coh}(X)$ and $\text{Vect}(X)$ the categories of coherent sheaves and vector bundles over $X$ respectively. The goal of this section is to prove the following result.

\begin{theorem}\label{myreticulo} Let $X$, $R$ and $x$ be as at the beginning of this section. Let furthermore $f:Y\to X_K$ be a finite $G$-torsor and $K'/K$ a finite extension. Consider the pullback $f_{K'}:Y'\to X_{K'}$ of $f$ and assume the existence of a morphism $\varphi:Z\to X$ such that $\mathcal{O}_Z(Z)=R'$, where $R'$ is the integral closure of $R$ in $K'$ and, as in figure \ref{figura1}, the restriction of $\varphi$ to the generic fibre coincides with $\lambda\circ f_{K'}$. Let $\mathcal{T}$ be the full subcategory of $\text{Coh}(X)$ whose objects are all the vector bundles $V \in \text{Vect}(X)$ trivialized by $\varphi$. If $\varphi$ is finite, faithfully flat and the locally free sheaf $\varphi_*(\mathcal{O}_Z)$ is trivialized by $\varphi$, then $\mathcal{T}$ endowed with the fibre functor $\omega := x^{*}$  and unit object $\mathcal{O}_X$ is a Tannakian lattice over $R$.
\end{theorem}
For simplicity we will use lower case letters to denote the rank of a vector bundle over a scheme, for example if $V \in \text{Vect}(X)$, we will write $v:=\text{rank}(V)$. The proof of Theorem \ref{myreticulo} involves the verification of a list of axioms, and that is why we will divide it in several propositions. Note that the condition $\varphi^{*}\varphi_{*}\mathcal{O}_Z$ is trivial implies that $\varphi_{*}\mathcal{O}_Z$ is an object of $\mathcal{T}$. The fact that  $\mathcal{O}_X(X)=R$ ensures that $\mathcal{T}$ is $R$-additive; it is also a rigid tensor category, in the common way, by taking the usual definitions of tensor product and dual of vector bundles.

In order to study kernels and images in our category we introduce a very simple technical lemma.

\begin{lemma}\label{plenfiel} If $\delta: T \rightarrow \text{Spec}(R)$ is a faithfully flat morphism that has a section $t: \text{Spec}(R) \rightarrow T$, and such that $\mathcal{O}_T(T)=R$, with $R$ being a PID, then the functor $\delta^{*}$ is fully faithful  over the free $\mathcal{O}_R$-modules.
\end{lemma}

\begin{proof}
Being $t$ a section of $\delta$, we have that $\delta \circ t = \text{id}_R$. In order to prove that $\delta^{*}$ is full, when restricted to free $\mathcal{O}_R$-modules, we first need to demonstrate that  $t^{*}$  is fully faithful over the free $\mathcal{O}_T$-modules. Let $m \in \text{Hom}(\mathcal{O}_R^{\oplus v}, \mathcal{O}_R^{\oplus w})$, note that $t^{*}\delta^{*}m= (\text{id}_R)^{*}m=m$, and therefore
\begin{equation}\label{eq5.1}
t^{*}: \text{Hom}_{\mathcal{O}_T}( \mathcal{O}_T^{\oplus v}, \mathcal{O}_T^{\oplus w}) \rightarrow \text{Hom}_{\mathcal{O}_R}(\mathcal{O}_R^{\oplus v}, \mathcal{O}_R^{\oplus w})
\end{equation}
is surjective. The first module is a free $\mathcal{O}_T(T)$-module of rank $vw$ and the second one is a free $R$-module also of rank $vw$, given that $\mathcal{O}_T(T)=R$. Then $t^{*}$  is actually an $R$-endomorphism of $R^{\oplus vw}$, now thanks to  \cite[Proposition 1.2, p. 506]{Vas68} we have that $t^{*}$ is an isomorphism. Now let $h: \mathcal{O}_T^{\oplus r} \rightarrow \mathcal{O}_T^{\oplus s}$ be a morphism of free  $\mathcal{O}_T$-modules. As $t$ is a section of $\delta$, then $t^{*}\delta^{*}t^{*}h= (\text{id}_R)^{*}t^{*}h = t^{*}h$, and given that  $t^{*}$ is faithful, this implies $
\delta^{*}t^{*}h= h$, and thus we have proven that $\delta^{*}$ is full. Exactly like before this is sufficient to deduce that \begin{align*}
\delta^{*}: \text{Hom}_{\mathcal{O}_R}(\mathcal{O}_R^{\oplus v}, \mathcal{O}_R^{\oplus w}) \rightarrow \text{Hom}_{\mathcal{O}_T}(\mathcal{O}_T^{\oplus v}, \mathcal{O}_T^{\oplus w})
\end{align*}
is an isomorphism, which concludes the proof.
\end{proof}

\begin{remark}\label{consLem2.4}
Notice that as a consequence of the proof of Lemma \ref{plenfiel}, one can prove that $t^*$, when restricted to free modules, is also exact. Indeed  $t^*(u)=t^*\delta^*(\lambda)=\text{id}_R^*(\lambda)$ for some $\lambda$, morphism of free $R$-modules, then $t^*(u)$ is injective if and only if $\lambda$ is injective, if and only if $u$ is injective, as $\delta$ is faithfully flat.
\end{remark}

\begin{proposition}\label{nucleoseimagenes}
The category $\mathcal{T}$ of Theorem \ref{myreticulo} has kernels and images.
\end{proposition}
\begin{proof}
Let $V,W \in \text{Obj}(\mathcal{T})$ and $\phi \in \text{Hom}_{\mathcal{T}}(V,W)$. We want to prove that both the kernel and image of $\phi$ are objects of $\mathcal{T}$. The morphism $\phi:V \rightarrow W$ induces a map $\varphi^{*}\phi: \varphi^{*}V \rightarrow \varphi^{*}W$, and given that $V,W$ are objects of $\mathcal{T}$, this, up to isomorphisms, can be seen as a morphism  $\varphi^{*}\phi: \mathcal{O}_Z^{\oplus v} \rightarrow \mathcal{O}_Z^{\oplus w}$ between free sheaves. Since the given $G$-torsor $f:Y \rightarrow X_K$ is Nori-reduced then we have that $\mathcal{O}_Y(Y)=K$, this implies that $\mathcal{O}_{Y^{\prime}}(Y^{\prime})=K'$ as cohomology commutes with flat base change. Now, as $\widetilde{j}$ is flat, we have an injection $\mathcal{O}_Z(Z) \rightarrow \mathcal{O}_{Y^{\prime}}(Y^{\prime})=K'$.

Now let $R'$ be the integral closure of $R$ in $K'$. Notice that the morphism $\varphi':= p \circ \varphi: Z \to \text{Spec}(R)$ is proper, as it is the  composition of proper morphisms. Then we consider the Stein factorization of $\varphi'$, given by $$\xymatrix{Z \ar[r]^-{\delta} & \Spec(R') \ar[r]^-{\mu} & \text{Spec}(R),}$$ for a proper morphism $\delta$ with geometrically connected fibres and a finite morphism $\mu$, which generically coincides with $\Spec(K')\to \Spec(K)$ \cite[Theorem 37.48.5]{Stack15}.

Hence we have a proper and faithfully flat map $\delta: Z \to \text{Spec}(R')$. In order to apply Lemma \ref{plenfiel}, we need a section to this map. Note that being $R$ complete, $R^{\prime}$ is also a complete discrete valuation ring, according to \cite[II, 2 Proposition 3]{Ser68}. Then from the valuative criterion of properness we deduce the existence of a section $t: \text{Spec}(R^{\prime}) \rightarrow Z$ extending $y_{K'}: \text{Spec}(K^{\prime}) \rightarrow Y$. 

The discussion above and Lemma \ref{plenfiel} imply that the morphism $\varphi^{*}\phi$ comes from a morphism of free $\mathcal{O}_{R^{\prime}}$-modules, that is, there exists $g: \mathcal{O}_{R^{\prime}}^{\oplus v} \rightarrow \mathcal{O}_{R^{\prime}}^{\oplus w}$ such that $\varphi^{*}\phi = \delta^{*}g$. Given that both the kernel and the image of $g$ are free $R^{\prime}$-modules, and using that $\delta^{*}$ is exact, we conclude that  $\text{Im}(\delta^{*}g) = \delta^{*}\text{Im}(g)= \text{Im}(\varphi^{*}\phi)$ and $\text{Ker}(\delta^{*}g)  =\delta^{*}\text{Ker}(g)= \text{Ker}(\varphi^{*}\phi)$, and that they are free $\mathcal{O}_Z$-modules. As $\varphi^*$ is an exact functor we have the isomorphisms
\begin{align*}
\text{Im}(\varphi^{*}\phi) \simeq \varphi^{*}\text{Im}(\phi) \text{ and } \text{Ker}(\varphi^{*}\phi) \simeq \varphi^{*}\text{Ker}(\phi),
\end{align*}
finally using (faithfully) flat descent we obtain that $\text{Im}(\phi)$ and $\text{Ker}(\phi)$ are locally free $\mathcal{O}_X$-modules. That they are trivialized by $\varphi$ follows directly from the exactness of $\varphi^{*}$.
\end{proof}
A consequence of the proof of the last two results is that, in our case, a map between free $\mathcal{O}_Z$-modules has kernel and image that are also free $\mathcal{O}_Z$-modules. Also the fact that $\text{Spec}(R')$ is the Stein factorization of $\varphi'=p\circ \varpi$, implies that the morphism $\varphi$ factors through $X_{R'}$ (the pullback of $X$ over $\text{Spec}(R')$), hence there is a finite morphism $h: Z \to X_{R'}$ such that $\varphi= \gamma \circ h$, where $\gamma$ is the projection $\gamma: X_{R'} \to X$. This situation will be of great use in \S \ref{sez:PrimetoP}.

Before continuing with the next step of the proof, we mention two results that will be useful later.
\begin{proposition}\label{leizhang} 
Assume that $X$ is either a reduced connected and proper scheme over a field $k$ such that  $\mathcal{O}_X(X)=k$, or a Nori-reduced finite torsor  over a reduced connected and proper scheme $T$ over a field $k$ such that  $\mathcal{O}_X(X)=k$. Assume moreover that $X$ has a $k$-rational point $x \in X(k)$, then for every essentially finite sheaf $V$ over $X$, the canonical morphism $V(X) \otimes_k \mathcal{O}_X \rightarrow V$, is an inclusion of $V(X) \otimes_k \mathcal{O}_X$ into the maximal trivial sub-sheaf of $V$. \end{proposition}
\begin{proof}
Though the setting is a bit different the proof is exactly the same as in \cite[Lemma 2.2]{Z13}, considering that a finite Nori-reduced torsor has a \emph{tannakian constructed} fundamental group scheme, \cite[Theorem I (3)]{ABETZ19}. 
\end{proof}
\begin{remark} From this result we can draw the following conclusion: if $V$ is a essentially finite sheaf of rank $v$ with global sections isomorphic to $k^{\oplus v}$, then this implies that $V \simeq \mathcal{O}_X^{\oplus v}$. Indeed because $V(X) \simeq k^{\oplus v}$ then $V(X) \otimes_k \mathcal{O}_X \simeq \mathcal{O}_X^{\otimes v}$ (this if $\mathcal{O}_X(X)=k$, which is true for $X$ as in the proposition), then there exists an immersion $\mathcal{O}_X^{\oplus v} \rightarrow V$ whose quotient is thus an essentially finite (hence locally free) sheaf of rank $0$, which implies that it must be necessarily the sheaf zero and therefore $V \simeq \mathcal{O}_X^{\oplus v}$.
\end{remark}
The next result, which will also prove to be useful, follows directly from \cite{TZ17}.
\begin{proposition}\label{tonini} \cite{TZ17}
Let $X$ be a proper, connected, normal and reduced scheme over a field $k$ and $V \in \text{Vect}(X)$, then $V$ is essentially finite if and only if there is a finite morphism $g:T \rightarrow X$ such that $V$ is trivialized by $g$.
\end{proposition}
Now, we deal with the category $\mathcal{T}_K$ that is obtained from $\mathcal{T}$ by extension of scalars to $K$, the field of fractions of $R$.
\begin{proposition}\label{Tkabeliana}
The category $\mathcal{T}_K$, which is the category obtained from $\mathcal{T}$ by extension of scalars, is abelian.
\end{proposition}

To prove this we introduce a new category, $\mathcal{T}_2$, whose objects are the vector bundles $V \in \text{Vect}(X_K)$ such that $(\lambda \circ f_{K^{\prime}})^{*}V \simeq \mathcal{O}_{Y^{\prime}}^{\oplus v}$, that is the vector bundles over $X_K$ that are trivialized by $\lambda \circ f_{K^{\prime}}= \varphi_K$.

\begin{lemma}\label{T2igualT3}
Let $\mathcal{T}_3$ be the category whose objects are the vector bundles over $X_K$ trivialized by the torsor $f$, then $\mathcal{T}_2= \mathcal{T}_3$.
\end{lemma}

\begin{proof}
As morphisms in both categories are clearly the same it is sufficient to prove that $\text{Obj}(\mathcal{T}_2) = \text{Obj}(\mathcal{T}_3)$. The inclusion $\text{Obj}(\mathcal{T}_3) \subseteq \text{Obj}(\mathcal{T}_2)$ is obvious. Conversely, let $V \in \text{Obj}(\mathcal{T}_2)$, meaning that $(\lambda \circ f_{K^{\prime}})^{*}V \simeq \mathcal{O}_{Y^{\prime}}^{\oplus v}$.
Being $\alpha$ flat, there is an isomorphism
\begin{align*}
f^{*}V(Y) \otimes_K K^{\prime} \simeq \alpha^{*}f^{*}V(Y^{\prime}).
\end{align*}
As $\alpha^{*}f^{*}V \simeq (\lambda \circ f_{K^{\prime}})^{*}V  \simeq \mathcal{O}_{Y^{\prime}}^{\oplus v}$, we obtain that $f^{*}V(Y) \simeq K^{\oplus v}$. Consequently we have the following
\begin{enumerate}
\item $f^{*}V$ has rank $v$.
\item $f^{*}V(Y) \simeq K^{\oplus v}$.
\item $f^{*}V \in \text{Obj}(\text{EF}(Y))$. [This is because $V \in \text{Obj}(\text{EF}(X_K))$ due to Proposition \ref{tonini}].
\end{enumerate}
Using Proposition \ref{leizhang} and the remark that followed we obtain $f^{*}V \simeq \mathcal{O}_Y^{\oplus v}$.
\end{proof}
\begin{remark}\label{T2abeliana}
The category $\mathcal{T}_3$  is the full subcategory of $\text{EF}(X_K)$ - the category of essentially finite vector bundles over $X_K$- generated by $f_{*}\mathcal{O}_Y$, which is tannakian and so, in particular, abelian. We will denote it also by $\text{EF}(X_K, \{ f_{*}\mathcal{O}_Y \})$. Note that if $V$ is an object of $\mathcal{T}$, then $j^{*}V$ is an object of $\mathcal{T}_2$, recalling that $\varphi_K= \lambda \circ f_{K'}$.
\end{remark}
We now recall a proposition to be used later.
\begin{proposition}\label{propA} \cite[Proposition 2.8.1]{EGAIV.2} Let $S$ be a Dedekind scheme with generic point $\eta$, $f:X \rightarrow S$ a morphism of schemes, $X_{\eta}$ the generic fibre, $i: X_{\eta} \rightarrow X$ the canonical morphism. Let $\mathcal{F}$ be a quasi-coherent $\mathcal{O}_X$-module, $\mathcal{F}_{\eta}:= i^{*}\mathcal{F}$, $\mathcal{G}^{\prime}$ a $\mathcal{O}_{X_{\eta}}$-module, quotient of $\mathcal{F}_{\eta}$ and $\overline{\mathcal{G}^{\prime}}$ the $\mathcal{O}_X$-module image of $\mathcal{F}$ obtained by means of the composition
\begin{align*}
\mathcal{F} \rightarrow i_{*}i^{*}\mathcal{F} \rightarrow i_{*}\mathcal{G}^{\prime}
\end{align*}
Then $\overline{\mathcal{G}^{\prime}}$ is a quasi-coherent and $S$-flat $\mathcal{O}_X$-module quotient of $\mathcal{F}$ such that $i^{*}(\overline{\mathcal{G}^{\prime}})=\mathcal{G}^{\prime}$, and also it is the only $\mathcal{O}_X$-module quotient of $\mathcal{F}$ with these properties. \end{proposition}

The following result ensures that, in our case, the quotient of the latter proposition is locally free:

\begin{lemma}\label{propB}
Let us take $V \in \text{Obj}(\mathcal{T}_K)$, set $V_K:= j^{*}V$ and let $Q \in \text{Obj}(\mathcal{T}_2)$ be a quotient of $V_K$ (through an epimorphism $u: V_K \rightarrow Q$). Then the only quotient $Q^{\prime}$ of $V$, $R'$-flat, such that $j^{*}Q^{\prime} \simeq Q$, is an object of $\mathcal{T}$ (then in particular it is locally free).
\end{lemma}
\begin{proof}
We denote by $u^{\prime}: V \rightarrow Q^{\prime}$ the epimorphism whose restriction to the generic fibre gives $u: V_K \rightarrow Q$. Pulling back $u: V_K \to Q$ over $Y'$ we obtain (see figure \ref{figura1}) the quotient   $$(f_{K^{\prime}})^{*}\lambda^{*}u: (f_{K^{\prime}})^{*}\lambda^{*}V_K \to (f_{K^{\prime}})^{*}\lambda^{*}Q.$$ Being $V \in \text{Obj}(\mathcal{T}_K)$ then $\varphi^{*}V \simeq \mathcal{O}_Z^{\oplus v}$, and therefore $\varphi_K^{*}V_K = (\lambda \circ f_{K^{\prime}})^{*}V_K \simeq \mathcal{O}_{Y^{\prime}}^{\oplus v}$. Given that $Q$ is an object of $\mathcal{T}_2$ then we have $(f_{K^{\prime}})^{*}\lambda^{*}Q \simeq \sO_{Y^{\prime}}^{\oplus s}$ for some positive integer $s$.

Observe now that $\varphi_K^*u = (f_{K'})^*\lambda^*u$ is a morphism of free $\sO_{Y'}-$modules, then, by Lemma \ref{plenfiel} it comes from a morphism $\gamma$ of free $K'$-modules. Because of \ref{propA} there exists a unique morphism $\gamma'$ of  free $R'$-modules such that $\gamma'_K= \gamma$. Now take $\delta^*\gamma: \varphi^*V \to \sO_Z^{\oplus s}$. Its restriction to $K'$ equals the morphism $\varphi*_Ku$, but also the morphism $\varphi^*u$ accomplishes this. Again by \ref{propA} this implies that $\sO_Z^{\oplus s}= \varphi^*Q'$, thus $Q'$ is locally free and a member of $\sT$.
\end{proof}
We now analyze kernels in the category $\mathcal{T}$.
\begin{lemma}\label{propC}
Let us take $V \in \text{Obj}(\mathcal{T}_K)$, set $V_K:= j^{*}V$ and let $W \in \text{Obj}(\mathcal{T}_2)$ be a subsheaf of $V_K$ (through a monomorphism $ v:W \to V_K$). Then there exists a monomorphism $v':W^{\prime} \to V$ for some  $W^{\prime} \in \text{Obj}(\mathcal{T}_K)$ such that  $j^{*}W^{\prime} \simeq W$. 
\end{lemma}
\begin{proof}
Consider the cokernel $u: V_K \rightarrow Q$ of $v$; as $V_K$ belongs to $\mathcal{T}_2$ and this category is abelian then $Q$ is an object of $\mathcal{T}_2$ too. Following Lemma \ref{propB} we construct the quotient $u^{\prime}: V \rightarrow Q^{\prime}$ which gives $u$ when restricted to the generic fibre, where  $Q^{\prime} \in \text{Obj}(\mathcal{T}_K)$. Set  $W^{\prime}:= \text{Ker}(u^{\prime})$, by applying $\varphi^{*}$ to the exact sequence
\begin{align*}
\xymatrix{
W^{\prime} \ar@{^{(}->}[r] & V \ar@{->>}[r]^{u^{\prime}} & Q^{\prime}
},
\end{align*}
we obtain another exact sequence
\begin{align*}
\xymatrix{
\varphi^{*}W^{\prime} \ar@{^{(}->}[r] & \varphi^{*}V \simeq \mathcal{O}_Z^{\oplus v} \ar@{->>}[r]^{\varphi^{*}u^{\prime}} & \varphi^{*}Q^{\prime} \simeq \mathcal{O}_Z^{\oplus q^{\prime}}
},
\end{align*}
in which $\varphi^{*}W^{\prime} \simeq \text{Ker}(\varphi^{*}u^{\prime})$, because of the exactness of $\varphi^{*}$. As mentioned immediately after proposition \ref{nucleoseimagenes}, kernels of maps between free $\mathcal{O}_Z$-modules are still free, hence $W^{\prime}$  is trivialized by $\varphi$, then, again by faithfully flat descent,  $W^{\prime}$ is locally free; moreover $W^{\prime} \in \text{Obj}(\mathcal{T}_K)$.
\end{proof}

\begin{proof}[Proof of the proposition  \ref{Tkabeliana}]
Let us consider the functor $F: \mathcal{T}_K \rightarrow \mathcal{T}_2$ given by
$$F(V):= j^{*}V,$$
for $V \in \text{Obj}(\mathcal{T}_K)=\text{Obj}(\mathcal{T})$. We will show that this functor is an equivalence of categories, and consequently because of Lemma \ref{T2igualT3} and Remark \ref{T2abeliana} we will obtain that $\mathcal{T}_K$ is abelian. The fact that $F$ is fully faithful follows clearly from the isomorphism
\begin{align*}
\text{Hom}_{\mathcal{T}_K}(V,W) := \text{Hom}_{\mathcal{T}}(V,W) \otimes_R K \simeq \text{Hom}_{\mathcal{T}_2}(j^{*}V,j^{*}W).
\end{align*} 
So we prove now that $F$ is essentially surjective. For this set $B:= f_{*}\mathcal{O}_Y$. Using the commutativity of direct image with flat base change we see that
$$
\lambda^{*}B= \lambda^{*}f_{*}\mathcal{O}_Y \simeq (f_{K'})_{*}\alpha^{*}\mathcal{O}_Y = (f_{K'})_{*}\mathcal{O}_{Y'},
$$
and therefore,
$$
\lambda_*\lambda^*B \simeq \lambda_*(f_{K'})_*\mathcal{O}_{Y'} =(\varphi_{K})_*\mathcal{O}_{Y'},
$$
the latter coming from the equality $\varphi_K = \lambda \circ f_{K'}$. Again using commutativity of direct image with flat base change we obtain that $(\varphi_{K})_*\mathcal{O}_{Y'} \simeq j^{*} \varphi_{*}\mathcal{O}_Z$. The natural immersion $B \to \lambda_*\lambda^*B$ thus gives us an immersion
$$
B \to j^{*}\varphi_{*}\mathcal{O}_Z.
$$
A hypothesis of Theorem \ref{teoGENERAL} was that $\varphi^* \varphi_*\sO_Z$ is a free $\sO_Z$-module, or in other words that $V:= \varphi_*\mathcal{O}_Z$ is an object of $\mathcal{T}$, in this case we have an immersion $B \to V_K$, where $V_K$ is an object of $\mathcal{T}_2$. By means of Lemma  \ref{propC}, there exists $W^{\prime} \in \text{Obj}(\mathcal{T}_K)$ such that $j^{*}W^{\prime} \simeq B$, consequently we have proven the essential surjectivity at the object $B$. By taking duals also we count with essential surjectivity in the dual $B^{\vee}$, and in finite direct sums of $B$ and $B^{\vee}$. Recalling that $\mathcal{T}_2= \mathcal{T}_3$, and this last category is spanned by the object $B$, so the only thing we are left to verify is that essential surjectivity is preserved through quotients and subquotients, however this is a consequence of Lemmas \ref{propB} and \ref{propC}.
\end{proof}

\begin{proposition}[Construction of the fibre functor $\omega$ for $\mathcal{T}$]\label{contruccionomega}
We set $\omega:= x^{*}$. Then $\omega$ is a $R$-linear additive functor
\begin{align*}
\omega: \mathcal{T} \rightarrow \text{Mod}(R)
\end{align*}
such that
\begin{itemize}
\item $\omega$ is faithful and preserves kernels and images.
\item $\omega$ restricted to $\mathcal{T}^{\text{triv}}$ is faithfully flat.
\end{itemize}
\end{proposition}

\begin{proof}
Consider the commutative diagram of functors
\begin{align*}
\xymatrix{
\mathcal{T} \ar[r]^{\omega} \ar[d]_{j^{*}} & \text{Mod}(R) \ar[d]^{- \otimes_R K} \\
\text{EF}(X_K) \ar[r]^{x_K^{*}} & \text{Vectf}_K
}
\end{align*}
Then $\omega= x^{*}$ is faithful as $j^{*}$ and $x_K^{*}$ are (indeed,  $x_K^{*}$ is the fibre functor of the category $\text{EF}(X_K)$, which is known to be faithful).

Now we show that $\omega$ preserves kernels and images: consider the map $\mu: \text{Spec}(R^{\prime}) \rightarrow \text{Spec}(R)$, which is a faithfully flat morphism (as $R'$ is a free $R$-module of finite rank). If we prove that $\mu^{*}x^{*} = (x \circ \mu)^{*}$ preserves kernels and images, then the same property will follow for $x^{*}$ because $\mu$ is faithfully flat. To see this let $u: V \rightarrow W$ be a morphism in $\mathcal{T}$, and let us consider its image factorization
\begin{align*}
\xymatrix{
V \ar@{->>}[r] &\text{Im}(u) \ar@{^{(}->}[r]& W,}
\end{align*}
if this is preserved by $\mu^{*}x^{*}$ then we have a factorization
\begin{align*}
\xymatrix{
\mu^{*}x^{*}V \ar@{->>}[r] &\mu^{*}x^{*}\text{Im}(u)=\text{Im}(\mu^{*}x^{*}u) \ar@{^{(}->}[r]& \mu^{*}x^{*}W,}
\end{align*}
and the fact that $\mu$ is faithfully flat induces the factorization 
\begin{align*}
\xymatrix{x^{*}V \ar@{->>}[r] &x^{*}\text{Im}(u)=\text{Im}(x^{*}u) \ar@{^{(}->}[r]& x^{*}W.}
\end{align*}
Hence $x^{*}$ preserves the image of $u$. A similar argument can be done with the kernel of $u$. Then we proceed to show that $(x \circ \mu)^{*}$ preserves kernels and images. Recall that in page 6 we used the valuative criterion of properness to conclude the existence of a section $t: \text{Spec}(R^{\prime}) \rightarrow Z$ of the morphism $\delta: Z \to \text{Spec}(R')$. Clearly $x \circ \mu = \varphi \circ t$, and therefore  $(x \circ \mu)^{*}= ( \varphi \circ t)^{*} = t^{*}\varphi^{*}$. Now let $u : V \rightarrow W$ be again a morphism of objects in $\mathcal{T}$ with image factorization
\begin{align*}
\xymatrix{
V \ar@{->>}[r] &\text{Im}(u) \ar@{^{(}->}[r]& W.}
\end{align*}
Now, $\varphi$ being a (faithfully) flat morphism, this factorization is preserved by $\varphi^{*}$, and therefore we have the factorization
\begin{align*}
\xymatrix{
\varphi^{*}V \ar@{->>}[r] &\text{Im}(\varphi^{*}u)=\varphi^{*}\text{Im}(u) \ar@{^{(}->}[r]& \varphi^{*}W,}
\end{align*}
However, as $V,W$ are objects of $\mathcal{T}$, then $\varphi^*V$ and $\varphi^*W$ are free $\mathcal{O}_Z$ modules, then so are its kernel and image, as in Proposition \ref{nucleoseimagenes}. Hence previous sequence is actually
\begin{align*}
\xymatrix{
\mathcal{O}_Z^{\oplus v} \ar@{->>}[r] &\mathcal{O}_Z^{\oplus s} \ar@{^{(}->}[r]& \mathcal{O}_Z^{\oplus w}.}
\end{align*}
Now by virtue of Remark \ref{consLem2.4} the previous sequence is sent by $t^{*}$ into
\begin{align*}
\xymatrix{
\mathcal{O}_{R^{\prime}}^{\oplus v} \ar@{->>}[r] &t^{*}\varphi^{*}(\text{Im}(u))=\mathcal{O}_{R^{\prime}}^{\oplus s} \ar@{^{(}->}[r]& \mathcal{O}_{R^{\prime}}^{\oplus w},}
\end{align*}
In conclusion the functor $t^{*}\varphi^{*}$ preserves images, and in a similar fashion it can be proven that it preserves kernels as well.

Now we prove that the restriction of $\omega$ to $\mathcal{T}^{\text{triv}}$ is fully faithful. The objects in $\mathcal{T}^{\text{triv}}$ are by definition those sheaves $V \in \text{Obj}(\mathcal{T})$ with a monomorphism $V \rightarrow \mathcal{O}_X^{\oplus n}$ for some $n \in \mathbb{N}$. By the proof of Lemma \ref{plenfiel} we already know that $x^{*}$- with $x$ being a section of $p: X \rightarrow \text{Spec}(R)$- is fully faithful at the free $\mathcal{O}_X$-modules, as $\mathcal{O}_X(X)=R$. So it is sufficient (and indeed necessary) to prove that the objects of $\mathcal{T}^{\text{triv}}$  coincide with the free $\mathcal{O}_X$-modules, and this will clearly provide us with the desired result.

Take  an object $V \in \text{Obj}(\mathcal{T})$ with a monomorphism $V \rightarrow \mathcal{O}_X^{\oplus n}$ for some $n \in \mathbb{N}$. Applying the functor $j^{*}$ and its exactness we obtain a monomorphism  $V_K \rightarrow \mathcal{O}_{X_K}^{\oplus n}$ in  $\text{EF}(X_K)$, where sub-objects of free modules are free as well so $V_K \simeq \mathcal{O}_{X_K}^{\oplus v}$. We can reduce to the case where $n=v$. Indeed, consider that the inclusion $V_K \to \sO_{X_K}^{\oplus n}$ has $\sO_{X_K}^{\oplus n-v}$ as its quotient (because $V_K \simeq \sO_{X_K}^{\oplus v}$), then by \ref{propA} the quotient of the inclusion $V \to \sO_X^{\oplus n}$ must be  equal to $\sO_X^{\oplus n-v}$, this gives us that the morphism $V\to \sO_X^{\oplus n} \to \sO_X^{\oplus n-v}$ equals zero. As $\sO_X^{\oplus v}$ is the kernel of the projection $\sO_X^{\oplus n}\to \sO_X^{\oplus n-v}$ then, by the definition of kernel, we must have an immersion $V \to \sO_X^{\oplus v}$. \\
Let $j_s: X_s \rightarrow X$ be the natural closed immersion induced by the special fibre of $X$, and $\varphi_s: Z_s \rightarrow X_s$ the restriction of $\varphi:Z \rightarrow X$ over the special fibre, and set $V_s:= j_s^{*}V$. As $V$ is by definition trivialized by $\varphi$ then we also have $\varphi_s^{*}V_s \simeq \mathcal{O}_{Z_s}^{\oplus v}.$
By virtue of Proposition \ref{tonini} that means that $V_s \in \text{Obj}( \text{EF}(X_s))$. By semicontinuity (see for instance \cite[III, Theorem 12.8]{Har77}) we have the following inequality
\begin{align*}
\text{dim}_k (V_s(X_s)) \geq \text{dim}_K (V_K(X_K))=v.
\end{align*}
However $V_s$ being essentially finite, Proposition \ref{leizhang} tells that we cannot have the strict inequality  $\text{dim}_k (V_s(X_s)) >v$, consequently $\text{dim}_k (V_s(X_s))=v$, and then by loc. cit. $V_s \simeq \mathcal{O}_{X_s}^{\oplus v}$. The equality
\begin{align*}
\text{dim}_k (V_s(X_s))= \text{dim}_K (V_K(X_K))=v
\end{align*}
implies, by Grauert's theorem  \cite[III, Corollary 12.9]{Har77}, that $p_{*}V$ is locally free over $\text{Spec}(R)$, then free, and generically isomorphic to $(p_K)_{*}V_K \simeq (p_K)_{*}\mathcal{O}_{X_K}^{\oplus v} \simeq \mathcal{O}_K^{\oplus v}$; then necessarily
$p_{*}V \simeq \mathcal{O}_R^{\oplus v}.$ Applying $p^{*}$ to both sides of the last isomorphism we obtain  $p^{*}p_{*}V \simeq \mathcal{O}_X^{\oplus v}$. Considering the canonical morphism $p^{*}(p_{*}V) \rightarrow V$, we obtain a morphism
\begin{align*}
\rho: \mathcal{O}_X^{\oplus v} \rightarrow V
\end{align*}
between $\mathcal{O}_X$-modules of the same rank which is generically an isomorphism. Then in particular it is injective. So it is sufficient to prove that it is surjective in order to obtain that $\rho$ is an isomorphism, which would conclude the proof.  
Again by Grauert's theorem the direct image base changes correctly, then the restriction of $\rho$ to the special fibre equals the natural morphism
\begin{align*}
(p_s)^{*}(p_s)_{*}V_s \rightarrow V_s,
\end{align*}
which is known to be an isomorphism when $V_s$ is a free sheaf of $\mathcal{O}_{X_s}$-modules, as in our case. So both fibres of $\rho$ are surjective, and that implies that $\rho$ itself is surjective (this can be better seen on the stalks).
\end{proof}

\section{Proof of Theorem \ref{teoGENERAL}}
\label{sec:proof}
In this section we provide the construction of the torsor $f_1: Y_1 \rightarrow X$, as claimed in Theorem \ref{teoGENERAL} or, more generally, a torsor over $X$ whenever we have a tannakian lattice $\mathcal{T}\simeq \Rep^0_R(M)$ over $X$. This problem has been faced in \cite[\S 2.4]{DH18} for $X$ an affine $R$-scheme and in \cite{DSH19} whenever the affine and $R$-flat group scheme $M$ associated to $\mathcal{T}$ is finite. In the latter it is simply suggested to follow Nori's construction presented in \cite[\S 2]{Nori76}. Our approach is similar to the second one just mentioned, but we are not assuming neither $X$ to be affine nor $M$ to be finite. In order to generalize Nori's \cite[Lemma 2.1]{Nori76} we need the following proposition, which is a particular case of a result due to Wedhorn in a much general setting.

\begin{proposition}\label{propWed}
Let $R$ be a Dedekind ring and $B$ a $R$-coalgebra. Let moreover $L$ be a $R$-flat $B$-comodule. Then $L$ is the directed limit of $B$-comodules which are finitely generated and projective over $R$ (then flat). 
\end{proposition}
\begin{proof}This is \cite[Corollary 5.10]{Wed04}.
\end{proof}

Now exactly like in the field case (as done in \cite[p. 31-32 Section 2]{Nori76}) we can build $\Rep^0_R(M)'$, the category of all $R$-linear (left) representations of $M$ which are directed limit of representations whose subjected module is finitely generated and projective. Then on a similar fashion, from the functor $F:\Rep^0_R(M)\simeq \mathcal{T}\to \text{Qcoh}(X)$ we obtain a new functor $F':\Rep^0_R(M)' \to \text{Qcoh}(X)$. In order to obtain a torsor over $X$, as in the field case, it is sufficient to consider $\mathbf{Spec}(F'(R[M]))$. The latter makes sense as $R[M]$ belongs to $\Rep^0_R(M)'$ by Proposition \ref{propWed} and it has of course an algebra structure. The details are left to the reader.

Therefore in our particular case the equivalence $F: \Rep_R(M) \simeq \mathcal{T}$ provides us with a $M$-torsor $f_1:Y_1 \to X$, where $Y_1:=\mathbf{Spec}(F'(R[M]))$. Remember that by scalar extension we obtain the category $\mathcal{T}_K$ for which we have proved, in section \ref{sec2.2}, the equivalence $\mathcal{T}_K \simeq \text{EF}(X_K, \{ f_{*}\mathcal{O}_Y \})\simeq \Rep_K(G)$ as $f:Y\to X$ is the given $G$-torsor. Then in particular $M_K\simeq G$ (as $(\Rep^0_R(M))_K \simeq  \Rep_K(M_K)$ by \cite[(6.4)]{Wed04}). What is left to verify is that the restriction of the $M$-torsor $f_1:Y_1 \to X$ to the generic fiber $X_K$ of $X$ is the original $G$-torsor $f: Y \to X_K$. This is essentially a consequence of \cite[Proposition 2.9]{Nori76} (mutatis mutandis). Given the equivalence $\Rep_R(M)\simeq \sT$, then the group scheme $M$ is affine because it is the group scheme associated to a tannakian lattice ($\sT$), the proof of this fact can be found in \cite[\S 2.3.1]{DH18}. It is also quasi-finite because its generic fiber $M_K \simeq G$ is finite (by hypothesis). This concludes the proof of Theorem \ref{teoGENERAL}.

\section{Application: prime to $p$ torsors}\label{sez:PrimetoP}

In what follows for any pro-finite group scheme $G$, by $G^{(p')}$ we will denote the biggest prime to $p$ quotient of $G$, that is the projective limit of all those finite quotients of $G$ whose order is prime to $p$.  We recall one of the most important consequences of the already mentioned ``th\'eorie de la sp\'ecialisation'' (cf. \cite{SGA}, Chapitre X, Th\'eor\`eme 3.8 and Corollaire 3.9) for the \'etale fundamental group here denoted by $\pi^{\text{\'et}}$:

\begin{theorem}\label{teoSPECIAL}
Let $f:X\to Y$ be a proper and smooth morphism of schemes, with geometrically connected fibres, with $Y$ locally noetherian, $s,\eta\in Y$ such that $s\in \overline{\{\eta\}}$, the latter denoting the Zariski closure of $\eta$ in $X$. We set $X_{\overline{s}}$ and $X_{\overline{{\eta}}}$ the respective geometric fibres. Then the specialization morphism $$sp:\pi^{\text{\'et}}(X_{\overline{{\eta}}},\overline{x}_1)\to \pi^{\text{\'et}}(X_{\overline{{s}}},\overline{x}_0)$$ (where $\overline{x}_1\in X_{\overline{{\eta}}}$ and $\overline{x}_0\in X_{\overline{{s}}}$ are geometric points) is surjective and:
\begin{itemize}
    \item if $char(k(s))=0$ then $sp$ is an isomorphism;
    \item if $char(k(s))=p>0$ then $$sp^{(p')}:\pi^{\text{\'et}}(X_{\overline{{\eta}}},\overline{x}_1)^{(p')}\to \pi^{\text{\'et}}(X_{\overline{{s}}},\overline{x}_0)^{(p')}$$ is an isomorphism.
\end{itemize}
\end{theorem}

Examples where $sp$ is not injective are well known. One of the main blocks in the proof of Theorem \ref{teoSPECIAL} is the following result (cf. also \cite[Theorem 5.7.10]{Sz09}): 

\begin{theorem}\label{teoSPECIAL1}
Let $R$ be a complete discrete valuation ring with algebraically closed residue field $k$. We set $S:=\Spec(R)$. Moreover let $f:X\to S$ be a proper and smooth morphism of schemes, with geometrically connected fibres and denote by $\eta$ and $s=\overline{s}$ the generic and special points of $S$ respectively. We set $X_s$ and $X_{\overline{\eta}}$ the special and generic geometric fibres respectively and we fix a geometric point $\overline{x}:\Spec(\Omega)\to X_{\overline{\eta}}$. Then the morphism 
$$\varphi: \pi^{\text{\'et}}(X_{\overline{\eta}},\overline{x})\to \pi^{\text{\'et}}(X,\overline{x})$$ is surjective and:
\begin{itemize}
    \item if $char(k)=0$ then $\varphi$ is an isomorphism;
    \item if $char(k)=p>0$ then $$\varphi^{(p')}:\pi^{\text{\'et}}(X_{\overline{\eta}},\overline{x})^{(p')}\to \pi^{\text{\'et}}(X,\overline{x})^{(p')}$$ is an isomorphism.
\end{itemize}
\end{theorem}

Same as for $sp$ it is known that $\varphi$ may not be injective, for instance when $X$ is a relative elliptic curve with supersingular special fibre.

While trying to extend the results of Theorem \ref{teoSPECIAL1} to Nori's fundamental group scheme one immediately notices that it is not possible to even compare the fundamental group schemes $\pi(X_{\overline{\eta}})$ and  $\pi(X)$ for a very simple reason: they are not just groups, they are group schemes over $k(\overline{\eta})$ and $R$ respectively,  so no group scheme morphism can exist between them. However we can at least build the following \emph{pro-finite} group scheme morphism
$$\psi: \pi(X_{\eta},x_{\eta})\to \pi(X,x)_{\eta}$$
where $x:\Spec(R)\to X$ is a section of $f$ and $x_{\eta}:\Spec(k({\eta}))\to X_{\eta}$ its restriction to $X_{\eta}$. The morphism $\psi$ being the schematic generalization of $\varphi$ it is natural to wonder whether it is a closed immersion or a faithfully flat morphism, when it is defined. Though $\psi$ is known to be always faithfully flat (cf. \cite[Proposition 5.5]{AEG20}), it is not known in general if it is a closed immersion (thus an isomorphism). It is known to be an isomorphism if $X$ is an abelian scheme (thus marking a first important difference with morphism $\varphi$) and in many cases if we restrict $\psi$ to the biggest abelian quotients (cf. \cite{A11}). Here we are going to prove that $\psi$ restricted to the prime to $p$ biggest quotients is an isomorphism, more precisely:

\begin{theorem}\label{teoFUND}
Let notations be as in Theorem \ref{teoSPECIAL1}, where moreover we endow $X$ with a section $x\in X(S)$ and assume $char(k)=p>0$ then the morphism:
$$\psi^{(p')}: \pi(X_{\eta},x_{\eta})^{(p')}\to \pi(X,x)_{\eta}^{(p')}$$
is an isomorphism.
\end{theorem}

The notation cannot lead to any ambiguity as the two operations of \emph{taking the prime to $p$ quotient} and \emph{reducing to the generic fiber} are interchangeable. That the morphism $\psi^{(p')}$ is faithfully flat is, again, a consequence of \cite[Proposition 5.5]{AEG20}. In order to prove that it is an isomorphism it is thus sufficient to prove that it is a closed immersion.

After this somewhat historical introduction, if we denote by $K$ the fraction field of $R$, then it may be useful to come back to previous notations for the generic fibre, i.e. $X_{\eta}=X_K$. Now take a finite Galois $G$-torsor $f:Y \rightarrow X_K$, where $X_K$ is the generic fibre of $X$, such that $\text{mcd}(p,|G|)=1$, with a point $y: \text{Spec}(K)\rightarrow Y$, over $x_K$. To prove that the morphism $\psi^{(p')}$ is a closed immersion it is enough to show that every such torsor $f: Y \rightarrow X_K$ can be extended to a $\widehat{G}$-torsor $f_1: Y_1 \rightarrow X$, for some finite and flat $R$-group scheme $\widehat{G}$, endowed with a point $\widehat{y} \in Y_1(R)$ over $x$, this is an easy consequence of \cite[Corollary 3.1]{A08}, as we already know that $\psi^{(p)}$ is faithfully flat. In order to prove this we will apply Theorem \ref{teoGENERAL} to Grothendieck's theory. Consequently, first of all we need to find a morphism $\varphi: Z \rightarrow X$ satisfying the hypotheses of the theorem. 

By Theorem \ref{teoSPECIAL1} we know that there exists a finite field extension $K\subseteq K'$ such that the $G_{K^{\prime}}$-torsor $Y'\to X_{K'}$ obtained as pull back of $Y\to X_K$ over $X_{K'}$ can be extended over $X^{\prime}:= X_{R'}$, where $R'$ is the integral closure of $R$ in $K'$. More precisely there exist a finite and flat $R'$-group scheme $H$ and a $H$-torsor $h: Z\to X_{R'}$, pointed over $x_{R'}$ (pull back of $x$), extending $Y'\to X_{K'}$.  As $\Spec(R')\to \Spec(R)$ is finite and faithfully flat (see for example \cite[II, 2 Proposition 3]{Ser68}), the same holds for $X_{R'}\to X$. In particular the composition $\varphi: Z \to X$ given by $\varphi:= \gamma \circ h$, where $\gamma: X^{\prime} \rightarrow X$ is the obvious morphism, is finite and faithfully flat too. Notice that by construction $Z$ is connected (being a Galois étale cover). The previous construction can be visualized in the following figure.
\begin{figure}[H]
    \centering
    \begin{align*}
    \xymatrixrowsep{0.05in}
\xymatrix{
Y^{\prime} \ar@{.>}[rr] \ar[rd] \ar[dd]_{f_{K^{\prime}}} & & Z \ar@{.>}[dd]_{h} & \\
& Y \ar[dd]_(.3){f}& & \\
X_{K^{\prime}} \ar[rr]|(.5)\hole \ar[dd]_{p_{K^{\prime}}}  \ar[rd]_{\lambda}& & X^{\prime} \ar[rd]_{\gamma} \ar[dd] & \\
& X_{K} \ar[dd]_(.3){p_K} \ar[rr] & & X \ar[dd]^{p} \\
\text{Spec}(K^{\prime}) \ar[rr]|(.5)\hole\ar[rd] & & \text{Spec}(R^{\prime}) \ar[rd] & \\
& \text{Spec}(K) \ar[rr] & & \text{Spec}(R)
}
\end{align*}
\caption{Model of a prime to $p$ torsor under a finite extension of $K$.}
    \label{figuraimportante}
\end{figure}
Summarizing, we have a finite and faithfully flat morphism $\varphi: Z \rightarrow X$  whose generic fibre equals $\lambda \circ f_{K^{\prime}}$. In this case we also have $\sO_Z(Z)= R'$. Indeed, the Galois condition on the given torsor tells us that $\sO_{Y'}(Y')=K'$, so we have an immersion $\sO_Z(Z) \to K'$, then the result follows because $\sO_Z(Z)$ is finite (and therefore integral) over $R'$ and because $R'$ is the integral closure of $R$ in $K'$. Therefore the only condition left to verify in order to apply Theorem \ref{teoGENERAL} is that $\varphi^{*}\varphi_{*}\mathcal{O}_Z$ is a free $\sO_Z$-module. We prove this in the next lemma.
\begin{lemma}\label{claimstar} Set $D:= h_{*}\mathcal{O}_Z$ then $\gamma_{*}D$ is trivialized by $\varphi$.
\end{lemma}
\begin{proof}
First we notice that, $k$ being algebraically closed, there is an isomorphism of special fibres $X_s \simeq X_s^{\prime}$, as the residue field of $R^{\prime}$ is equal to $k$, same as the residue field of $R$.
This implies (cf. \cite[Expos\'e X, theoreme 2.1]{SGA}), that there is an isomorphism of \'etale fundamental groups $\pi^{\text{\'et}}(X,\overline{x}) \simeq \pi^{\text{\'et}}(X^{\prime}, \overline{x})$, where $\overline{x}$ is a geometric point taken with respect to $\overline{K}$, an algebraic closure of $K$. Given that $h$ is actually a (connected) finite \'etale Galois covering of $X^{\prime}$, the above isomorphism of \'etale groups implies the existence of a (connected) finite \'etale galois covering $\overline{h}: T \rightarrow X$, with pullback over $R^{\prime}$ isomorphic to $Z \rightarrow X^{\prime}$, this is because of \cite[Corollary 5.5.8]{Sz09}. Then we have the following Cartesian squares
\begin{align*}
\xymatrix{
Z \ar[d]_h \ar[r]^{\overline{\gamma}} & T \ar[d]^{\overline{h}} \\
X^{\prime} \ar[r]^{\gamma} \ar[d]^{p^{\prime}} & X \ar[d]^p\\
\text{Spec}(R^{\prime}) \ar[r]^{\beta} &\text{Spec}(R).
}
\end{align*}
 Recall that $R^{\prime}$ is a free $R$-module of rank $n=[K^{\prime}:K]$, this induces the isomorphism of $\mathcal{O}_R$-modules  $\beta_{*}\mathcal{O}_{R^{\prime}} \simeq \mathcal{O}_R^{\oplus n}$. Using commutativity of the direct image with flat base change in the lower square we have 
\begin{align*}
\overline{\gamma}_{*}\mathcal{O}_Z = \overline{\gamma}_{*}h^{*}(p^{\prime})^{*}\mathcal{O}_{R^{\prime}} \simeq \overline{h}^{*}p^{*}\beta_{*}\mathcal{O}_{R^{\prime}} \simeq \overline{h}^{*}p^{*} \mathcal{O}_R^{\oplus n} \simeq \mathcal{O}_{T}^{\oplus n}.
\end{align*} 
But $h:Z \rightarrow X^{\prime}$ is a torsor, then we have  $h^{*}h_{*}\mathcal{O}_Z \simeq \mathcal{O}_Z^{\oplus g}$, where $g = |H|= |G|$, in other words  $h^{*}D \simeq \mathcal{O}_Z^{\oplus g}$. Again using flat base change we obtain that
\begin{align*}
\overline{h}^{*}\gamma_{*}D \simeq \overline{\gamma}_{*}h^{*}D \simeq \overline{\gamma}_{*}\mathcal{O}_Z^{\oplus g} \simeq \mathcal{O}_{T}^{\oplus gn}.
\end{align*}
Applying $\overline{\gamma}^{*}$ to the latter we obtain
\begin{align*}
\overline{\gamma}^{*}\overline{h}^{*}\gamma_{*}D \simeq \mathcal{O}_{Z}^{\oplus gn},
\end{align*}
but $\overline{\gamma}^{*}\overline{h}^{*} = (\gamma \circ h)^{*} = \varphi^{*}$, in other words $\varphi^{*}\gamma_{*}D \simeq \mathcal{O}_{Z}^{\oplus gn}$, as desired.
\end{proof}

It is interesting to point out that in this case $M$ is finite, this is why since the beginning of this section we considered the \emph{classical} pro-finite fundamental group scheme over $X$ instead of its quasi-finite counterpart. This will conclude the proof of Theorem \ref{teoFUND}. We prove this fact in the following.

\begin{proposition}
The $R$-group scheme $M$ built above is finite.
\end{proposition}

\begin{proof}
In this particular case the main feature is that the finite morphism $h: Z \to X_{R'}$ is already a $H$-torsor for some \emph{finite} group scheme $H$ over $R'$. The equivalence $\sT \simeq \Rep^0_R(M)$ induces an equivalence $\sT_{R'} \equiv \Rep^0_{R'}(M_{R'})$. If $\sH$  is the full subcategory of $\text{Coh}(X_{R'})$ formed by vector bundles trivialized by the $H$-torsor $h$, then we can consider the functor $\gamma^*: \sT_{R'} \to \sH$, that is faithful, because $\gamma$ is faithfully flat. Using \cite[Lemma 2.11]{AE18} we obtain an equivalence $\sH \simeq \Rep^0_{R'}(H)$, and therefore we obtain a functor $\Rep^0_{R'}(M_{R'}) \to \Rep^0_{R'}(H)$, and thus an affine group morphism $H \to M_{R'}$ whose generic fibre $H_{K'} \to (M_{R'})_{K'}$ is an isomorphism. This induces in turn an immersion $R'[M_{R'}] \to R'[H]$, as both are $R'$-flat modules. From the fact that $R'$ is Noetherian it follows that $R'[M_{R'}]$ is a free $R$-module of finite rank too and so is $R[M]$.
\end{proof}

The reader will notice that when the assumptions are comparable the latter is also a consequence of \cite[Theorem 8.10]{HdS19}.

\textsc{Universidad de Costa Rica, Ciudad universitaria Rodrigo Facio Brenes,
Costa Rica.}\\
\textit{E-mail address:} marco.antei@ucr.ac.cr\\

\textsc{Universidad de Costa Rica, Ciudad universitaria Rodrigo Facio Brenes,
Costa Rica.}\\
\textit{E-mail address:} jimmy.calvo@ucr.ac.cr

\end{document}